\documentclass{amsart}
\usepackage{epsfig, amssymb, amsfonts, color, mathrsfs, enumitem}
\usepackage[all]{xy}
\usepackage[margin=1in]{geometry}

\newtheorem{lemma}{Lemma}[section]
\newtheorem{prop}[lemma]{Proposition}
\newtheorem{thm}[lemma]{Theorem}

\newtheorem{conj}[lemma]{Conjecture}
\newtheorem{cor}[lemma]{Corollary}
\newtheorem{defn}[lemma]{Definition}
\newtheorem{lem}[lemma]{Lemma}
\newtheorem{rem}[lemma]{Remark}

\newtheorem*{special theorem}{My Specially-Named Theorem}

\newcommand{\Z} { {\mathbb Z} }
\newcommand{\Q} { {\mathbb Q} }

\newcommand{\colim} { {\mathrm{colim}} }
\newcommand{\Hom} { {\mathrm{Hom}} }
\newcommand{\et} { {\acute et} }

\newcommand{\DM} { {\mathrm{DM}} }
\newcommand{\eff} { {\mathrm{eff}} }

\newcommand{\inthom} { {\underline{\mathrm{Hom}}} }
\newcommand{\HI} { {\mathrm{HI}} }
\newcommand{\tr} { {\mathrm{tr}} }
\newcommand{\Sm} { {\mathrm{Sm}} }
\newcommand{\Sch} { {\mathrm{Sch}} }
\newcommand{\Cor} { {\mathrm{Cor}} }
\newcommand{\PST} { {\mathrm{PST}} }
\newcommand{\Sh} { {\mathrm{Sh}} }
\newcommand{\kk} { {\boldsymbol{k}} }

\newcommand{\comment}[1]{}

\newcommand{\skipline}{\vspace{12pt}}

\title{Filtrations on homotopy invariant sheaves with transfers}
\date{}
\author{Tohru Kohrita}
\address{Freie Universit\"at Berlin, Arnimallee 3, 14195 Berlin, Germany}
\email{tohru.kohrita@fu-berlin.de}

\begin{document}

\maketitle

\begin{abstract}
We construct filtrations on homotopy invariant sheaves with transfers and show that under Ayoub's conjectures on $n$-motives, our filtration agrees with the one conjectured by Ayoub and Barbieri-Viale if the latter exists. Our construction is directly motivated by the work of Pelaez.
\end{abstract}


\section{Introduction} 

Let $\DM^\eff(\kk,\Q)$ be Voevodsky's triangulated category of unbounded effective motives over a perfect field $\kk$ with rational coefficients. The triangulated subcategory $\DM_{\leq n}^\eff(\kk,\Q)$ of $n$-motives is defined as the smallest localizing subcategory of $\DM^\eff(\kk,\Q)$ that contains the motives of smooth schemes of dimension $\leq n.$ While the inclusion $\DM^\eff_{\leq n}(\kk,\Q)\longrightarrow \DM^\eff(\kk,\Q)$ has a right adjoint for an arbitrary $n$ for abstract reasons, left adjoints exist only for $n=0$ and $1,$ but they are closely related to the theory of the algebraic equivalence relation and Albanese varieties. The existence of left adjoints for $n\geq2$ leads to a contradiction as observed in \cite[2.5]{ABV}.

On the other hand, one may focus on the heart of $\DM^\eff(\kk,\Q)$ with respect to the homotopy $t$-structure. The heart is known to be equivalent to the category $\HI_\et^\tr(\kk,\Q)$ of homotopy invariant \'etale sheaves of $\Q$-modules with transfers. In this category, Ayoub and Barbieri-Viale (\cite{ABV}) defined cococmplete abelian subcategories $\HI^\tr_{\et,\leq n}(\kk,\Q)$ of $n$-motivic sheaves for non-negative integers $n$ (see Subsection~\ref{subsection: $n$-motivic sheaves} for a quick review). It is conjectured that $\HI^\tr_{\et,\leq n}(\kk,\Q)$ is the heart of $\DM^\eff_{\leq n}(\kk,\Q)$ with respect to (the restriction of) the homotopy $t$-structure (\cite[Conjecture 4.27]{AyoubOpen}, \cite[Conjecture 2.5.3]{ABV}). This conjecture is affirmative for $n=0$ and $1$ (\cite{VSF5, Orgogozo, Barbieri-Viale-Kahn, ABV}).

Now, let $h_0^\et(X)_\Q$ be the $0$-th cohomology of the motive $M(X)$ of a smooth $\kk$-scheme $X$ with respect to the homotopy $t$-structure. More explicitly, it is the \'etale sheafification of the $0$-th homology of the Suslin complex $C_*\Q_{\tr}(X).$\comment{

it is the \'etale sheafification of the presheaf on the category $\Sm/\kk$ of smooth $\kk$-schemes that takes the value $\Hom_{\DM^\eff(\kk,\Q)}(M(U), M(X))$ at $U\in \Sm/\kk.$} It is shown in \cite{ABV} that if the sheaves $h_0^\et(X)_\Q$ carry filtrations satisfying a certain set of axioms (see Conjecture~\ref{conj: ABV}), then the inclusions $\HI^\et_{\leq n}(\kk,\Q)\longrightarrow \HI^\et(\kk,\Q)$ admit left adjoints for all $n.$ In \emph{ibid.,} it was remarked that when $X$ is smooth and projective, the filtration induced on $h^\et_0(X)_\Q(K)\cong CH_0(X\times_kK)\otimes\Q$ ($K$ is any finitely generated field over $\kk$) should agree with the conjectural Bloch-Beilinson filtrations on $0$-cycle Chow groups.

The purpose of this article is to construct a filtration on $h_\tau^0(X)_R$ for $\tau\in\{Nis, \et\}$ with $R$ a commutative ring in which the characteristic of $\kk$ is invertible, and show that this filtration is a candidate for the Ayoub-Barbieri-Viale filtration in the sense of Theorem~\ref{thm}. By construction, our filtration agrees with that of Pelaez (\cite{Pelaez}) when evaluated at finitely generated fields over $k$ for $\tau=Nis$ or for a $\Q$-algebra $R.$ The relation with the Bloch-Beilinson filtration is provided by the following: Using the work of Voisin (\cite[Proposition 6]{Voisin}) under the Lefschetz standarad conjecture, Pelaez (\cite[6.1.9]{Pelaez}) showed that his filtration for $0$-cycle Chow groups is contained in the conjectural Bloch-Beilinson filtration.

\skipline
{\it Notation and conventions.}

We assume that the base field $\kk$ is perfect and of exponential characteristic $p,$ and schemes are separated and of finite type over $\kk.$ \comment{Our arguments depend on, among other things, Voevodsky's cancellation theorem (\cite{Voevodsky cancellation}). This requires the perfectness of the base field.}The category of schemes (resp., smooth schemes) over $\kk$ with $\kk$-morphisms is denoted by $\Sch/\kk$ (resp., $\Sm/\kk$). 

The symbol $\inthom^\eff$ stands for the internal hom in $\DM_{Nis}^\eff(\kk,R)$ for arbitrary $R$ or that in $\DM_{\et}^\eff(\kk,R)$ for a $\Q$-algebra $R.$ This shall not cause confusions as the two categories are equivalent as tensor triangulated categories when $R$ is a $\Q$-algebra. The definitions of triangulated category of motives are recalled in Section~\ref{section: The filtration}.

Assumptions stated at the beginning of a section (resp., subsection) run through the section (resp., subsection).


\section{The conjecture of Ayoub and Barbieri-Viale}\label{section: The conjecture of Ayoub and Barbieri-Viale}

We summarize necessary facts on $n$-motivic sheaves and recall the conjecture of Ayoub and Barbieri-Viale \cite[Conjecture 1.4.1]{ABV}. We assume that the base field $\kk$ is perfect and of exponential characteristic $p,$ and $R$ denotes the ring of coefficients. In this section, $\tau\in\{triv, Nis, \et\},$ where $triv$ stands for the trivial topology.


\subsection{Homotopy invariant sheaves with transfers}\label{subsection: Homotopy invariant sheaves with transfers}

Let $\Cor(\kk)$ be the category of finite correspondences over $\kk$ as in \cite[Lecture 1]{MVW}. The category $\PST(\kk, R)$ of {\bf presheaves of $R$-modules with transfers} on $\Sm/\kk$ is the category of contravariant additive functors from $\Cor(\kk)$ to the category $R\mathrm{Mod}$ of $R$-modules. $R_\tr(-)\colon\Cor(\kk)\longrightarrow \PST(\kk,R)$ denotes the embedding given by $R_\tr(X)=\Hom_{\Cor(\kk)}(-,X)\otimes_\Z R.$ A presheaf with transfers is called a {\bf $\tau$-sheaf with transfers} on $\Sm/\kk$ if it is a $\tau$-sheaf when restricted to $\Sm/\kk.$ The full subcategory of sheaves with transfers in $\PST(\kk, R)$ is denoted by $\Sh_\tau^\tr(\kk,R).$ A presheaf with transfers of the form $R_\tr(X)$ is in fact a $\tau$-sheaf with transfers (\cite[Lemma 6.2]{MVW}).

A sheaf with transfers $F$ is said to be {\bf homotopy invariant} if, for any $X\in\Sm/\kk,$ the morphism $F(X)\longrightarrow F(X\times_\kk \mathbb A_\kk^1)$ induced by the projection $X\times_\kk \mathbb A_\kk^1\longrightarrow X$ is an isomorphism. The full subcategory of homotopy invariant $\tau$-sheaf with transfers in $\Sh_\tau^\tr(\kk, R)$ is denoted by $\HI_\tau^\tr(\kk, R).$ It is classical (\cite[Lemma 4.2]{Swan}) that the inclusion $\HI_{triv}^\tr(\kk, R)\longrightarrow\Sh_{triv}^\tr(\kk, R)=\PST(\kk, R)$ has a left adjoint $h_0^{triv}$ given by $h_0^{triv}(F)=H_0(C_*F),$ where $C_*F$ is the singular simplicial complex of $F$ as in \cite[3.2]{VSF5}. For details, we refer the reader to \emph{loc. cit.} or \cite[Lecture 2]{MVW}. More generally, we have the following.

\begin{prop}[Suslin, Voevodsky]
Let $\tau\in\{triv, \et, Nis\}$ and $R$ be a ring. If $\tau=\et,$ assume that $p$ is invertible in $R.$ Then, the inclusions $\HI_\tau^\tr(\kk,R)\longrightarrow \Sh_\tau^\tr(\kk, R)\buildrel i\over\longrightarrow \PST(\kk,R)$ admit left adjoints $\PST(\kk,R)\buildrel a_\tau\over\longrightarrow \Sh_\tau^\tr(\kk,R)\buildrel h_0^\tau\over\longrightarrow \HI_\tau^\tr(\kk,R).$ Here, $a_\tau$ is given by the $\tau$-sheafification and $h_0^\tau$ by the composition $a_\tau\circ h_0\circ i.$ 
\end{prop}

\begin{proof}
This is part of \cite[Lemma 1.1.1 and Proposition 1.1.2]{ABV}. (When $\tau=\et,$ $p$ needs to be invertible because the proof depends on Suslin's rigidity theorem.)
\end{proof}

Let us introduce the sheaf of our main interest.

\begin{defn}
For $X\in \Sm/\kk$ and a ring $R,$ a $\tau$-sheaf $h_0^\tau(X)_R$ is defined as $h_0^\tau(X)_R=h_0^\tau(R_\tr(X)).$
\end{defn}

This sheaf is closely related to Suslin homology, which, for proper schemes, is nothing but Chow groups modulo rational equivalence.

\begin{prop}
Let $X\in\Sm/\kk.$ If $\tau\in\{triv, Nis\},$ then $h_0^\tau(X)(\kk)$ is canonically isomorphic to the $0$-th Suslin homology $H_0^S(X,\Z).$ If $\tau=\et,$ this is still true rationally: $h_0^\et(X)_\Q(\kk)\cong H_0^S(X,\Q).$
\end{prop}

\begin{proof}
Let $R$ be a ring. The Suslin homology of $X$ is defined as $H_n^S(X,R)=H_n(C_*(\Z_\tr(X)\otimes_\Z R)(\kk)).$ Thus, if $\tau=triv,$ there is nothing to prove. If $\tau=Nis,$ the statement follows because a field does not have nontrivial Nisnevich coverings. The case $\tau=\et$ is also true because any Nisnevich sheaf of $\Q$-modules with transfers is a sheaf in the \'etale topology as well (\cite[Corollary 14.22]{MVW}).
\end{proof}


\subsection{$n$-motivic sheaves}\label{subsection: $n$-motivic sheaves}

Let $(\Sm/\kk)_{\leq n}$ (resp., $\Cor(\kk_{\leq n})$) be the full subcategory in $\Sm/\kk$ (resp., $\Cor(\kk)$) that consists of schemes of dimension at most $n.$ Endow $(\Sm/\kk)_{\leq n}$ with the $\tau$-topology. Note that any $\tau$-covering of a scheme $X$ has the same dimension as $X.$ We define the category $\PST(\kk_{\leq n},R)$ of {\bf presheaves with transfers} on $(\Sm/\kk)_{\leq n}$ as the category of contravariant additive functors from $(\Sm/\kk)_{\leq n}$ to $R\mathrm{Mod}.$ A presheaf with transfers on $(\Sm/\kk)_{\leq n}$ is called a $\tau$-sheaf with transfers on $(\Sm/\kk)_{\leq n}$ if it is a $\tau$-sheaf when restricted to $(\Sm/\kk)_{\leq n}.$ We write $\Sh_\tau^\tr(\kk_{\leq n},R)$ for the full subcategory in $\PST(\kk_{\leq n},R)$ of $\tau$-sheaves with transfers on $(\Sm/\kk)_{\leq n}.$

The exact functor $\sigma_{\leq n*}\colon \Sh_\tau^\tr(\kk,R)\longrightarrow \Sh_\tau^\tr(\kk_{\leq n},R)$ induced by the inclusion $\sigma_n\colon (\Sm/\kk)_{\leq n}\longrightarrow \Sm/\kk$ has a left adjoint $\sigma_n^* \colon \Sh_\tau^\tr(\kk_{\leq n},R)\longrightarrow \Sh_\tau^\tr(\kk,R),$ which is given by $\sigma_n^*(F)=\colim_{X\rightarrow F} R_\tr(X).$ Here, the colimit is computed in $\Sh_\tau^\tr(\kk,R)$  and the index category is the category $\Cor(\kk_{\leq n})/F$ whose objects are arrows $R_\tr(X)\to F$ in $\Sh_\tau^\tr(\kk_{\leq n},R)$ with $X\in \Cor(\kk_{\leq n})$ and morphisms are given by commutative diagrams 
\begin{displaymath}
\xymatrix{ R_\tr(X) \ar[dr] \ar[rr] && R_\tr(Y)  \ar[dl] \\
& F}
\end{displaymath}
of $\tau$-sheaves with transfers (\cite[Lemma 1.1.12]{ABV}).

\begin{defn}[{\cite[Definition 1.1.20]{ABV}}]
A homotopy invariant sheaf $F\in \HI_\tau^\tr(\kk,R)$ is {\bf $n$-motivic} if the counit of the adjunction $\sigma_n^*\dashv\sigma_{n*}$ induces an isomorphism $h_0^\tau(\sigma_n^* \sigma_{n*}(F))\buildrel\cong\over\longrightarrow h_0^\tau(F).$ 

The full subcategory of $n$-motivic $\tau$-sheaves in $\HI_\tau^\tr(\kk,R)$ is denoted by $\HI_{\tau,\leq n}^\tr(\kk,R).$ 
\end{defn}

\begin{rem}\label{rem: n-generated sheaves}
A sheaf $F\in Sh_\tau^\tr(\kk, R)$ is called {\bf $n$-generated} (resp., {\bf strongly $n$-generated}) if the counit $\sigma_n^*\sigma_{n,*}F\longrightarrow F$ is a surjection (resp., isomorphism). Any $n$-motivic $\tau$-sheaf is the $h_0^\tau$ of a strongly $n$-generated $\tau$-sheaf, and conversely, $h_0^\tau$ of any strongly $n$-generated $\tau$-sheaf is $n$-motivic. In particular, if $X\in(\Sm/\kk)_{\leq n},$ then $h_0^\tau(X)$ is an $n$-motivic $\tau$-sheaf. See \cite[Remark 1.1.21]{ABV} for the proof of these. 
\end{rem}

For $n$-motivic sheaves, we generally know the following. (See Remark~\ref{rem: conjectural world} for what is conjecturally expected.)
 
\begin{prop}[{\cite[Lemma 1.1.22, Corollary 1.1.24]{ABV}}]\label{prop: proved facts on n-motivic sheaves}
Let $\tau\in\{triv, \et, Nis\}.$ Assume that $p$ is invertible in $R$ when $\tau=\et.$ Then,
\begin{enumerate}
\item the property of being $n$-motivic is stable under taking cokernels and extensions in $\HI_\tau^\tr(\kk,R).$
\item The category $\HI_{\tau,\leq n}^\tr(\kk,R)$ is abelian and cocomplete, and the inclusion  $\HI_{\tau,\leq n}^\tr(\kk,R)\longrightarrow  \HI_{\tau}^\tr(\kk,R)$ is right exact.
\end{enumerate}
\end{prop}


\subsection{The conjectures}

In this subsection, $R$ denotes a ring in which $p$ is invertible. The conjecture is concerned with \'etale sheaves $h_0^\et(X)_R$ of $R$-modules.

\begin{conj}[{\cite[Conjecture 1.4.1]{ABV}}]\label{conj: ABV}
For any $X\in\Sm/\kk,$ there exists a decreasing filtration $F^{n}h_0^\et(X)_R\supset F^{n-1}h_0^\et(X)_R$ such that 
\begin{enumerate}[label=(\Alph*)]
\item $F^{0}h_0^\et(X)_R=h_0^\et(X)_R$ and $F^{n}h_0^\et(X)_R=0$ for $n\geq\dim X+1.$
\item The filtration is compatible with the action of correspondences, i.e. for $\gamma\in Cor(X,Y),$ the induced morphism of sheaves $h_0^\et(X)_R\longrightarrow h_0^\et(Y)_R$ is compatible with the filtration. 
\item If $U$ is a dense open subscheme of $X,$ then $h_0^\et(U)_R\longrightarrow h_0^\et(X)_R$ is strict for the filtration.
\item For $n\geq 0,$ the quotient $h_0^\et(X)_R/F^{n+1}h_0^\et(X)_R$ is $n$-motivic.
\end{enumerate}
\end{conj}

(Axiom (D) actually follows from (A), (B) and a \emph{Weaker Version} of (D): \emph{For $n\geq 0,$ the quotient $h_0^\et(X)_R/F^{n+1}h_0^\et(X)_R$ is $n$-generated}; see \cite[Lemma 1.4.3]{ABV}.)

\begin{rem}[cf. Proposition~\ref{prop: proved facts on n-motivic sheaves}]\label{rem: conjectural world}
Under (A), (B) and (D), \cite[Corollary 1.4.5]{ABV} states that $\HI_{\et,\leq n}^\tr(\kk,R)$ is a Serre subcategory of $\HI_{\et}^\tr(\kk,R)$ (i.e. closed under subobjects, quotients and extensions) and the inclusion $\HI_{\et,\leq n}^\tr(\kk,R)\longrightarrow  \HI_{\et}^\tr(\kk,R)$ is exact. For $n=0$ and $1,$ this result is unconditionally proved in \emph{[\emph{ibid.,} Proposition 1.2.7, Corollary 1.3.5]}.
\end{rem}

Another beautiful consequence of Conjecture~\ref{conj: ABV} is the following (\cite[Proposition 1.4.6]{ABV}): If the conjecture is true for $R=\Q,$ then the inclusion $\HI_{\et,\leq n}^\tr(\kk,\Q)\longrightarrow  \HI_{\et}^\tr(\kk,\Q)$ admits a left adjoint $(-)^{\leq n}\colon \HI_{\et}^\tr(\kk,\Q)\longrightarrow  \HI_{\et,\leq n}^\tr(\kk,\Q)$ for an arbitrary $n.$ As we have explained in the Introduction, this is in contrast to the derived situation where the existence of left adjoints of the inclusion $\DM_{\et,\leq n}^\eff(\kk,\Q)\longrightarrow \DM_{\et}^\eff(\kk,\Q)$ for $n\geq 2$ leads to a contradiction at least when $\kk$ is algebraically closed and has infinite transcendence degree over $\Q$ (\cite[2.5]{ABV}). It is also known that, conversely, if the left adjoints $(-)^{\leq n}$ exist and $\HI_{\et,\leq n}^\tr(\kk,\Q)$ are Serre subcategories of $\HI_{\et}^\tr(\kk,\Q),$ then Conjecture~\ref{conj: ABV} holds. The filtration and the left adjoints are related by the equation $F^nh_0^\et(X)_\Q=\ker\{h_0^\et(X)_\Q\longrightarrow h_0^\et(X)_\Q^{\leq n-1}\};$ hence the filtration as in Conjecture~\ref{conj: ABV} is unique if it exists.


\section{The filtration on $h_0^\tau(X)_R$}\label{section: The filtration}

Our construction is motivated by the work of \cite{Pelaez}, especially Corollary 5.3.3 thereof. In this section, $\tau\in\{Nis, \et\}$ and we assume that $p$ is invertible in the coefficient ring $R$ in order to avoid the use of resolution of singularities and to ensure the existence of the homotopy $t$-structures on $\DM_\tau^\eff(\kk,R)$ for $\tau=\et.$ We write $\DM_\tau(\kk,R)$ for Voevodsky's triangulated category of unbounded $\tau$-motives with coefficients in $R.$ By this, we mean the homotopy category (with respect to the stable model structure defined in \cite[D\'efinition 4.3.29]{Ayoub}) of the model category $\mathrm{Spt}_T(\mathrm{Ch}(\mathrm{Sh}_\tau^\tr(\kk,R)))$ of symmetric $T$-spectra, where $\mathrm{Ch}(\mathrm{Sh}_\tau^\tr(\kk,R))$ is endowed with the $\mathbb A^1$-local model structrue, i.e. the model structure obtained by the Bousfield localization of the injective model structure with respect to the class of morphisms $C_{\mathbb A^1}=\{R_\tr(X\times_\kk\mathbb{A}_\kk^1)[n]\longrightarrow R_\tr(X)[n]| X\in\Sm/\kk, n\in\Z\}.$ We write $M\colon \Sm/\kk\longrightarrow \DM_\tau(\kk,R)$ for the canonical functor that associates smooth schemes with their motives.


\subsection{Construction of the filtration}\label{subsection: Construction of the filtration}

A triangulated category $\mathrm{T}$ with arbitrary coproducts is {\bf compactly generated} if there is a set $S$ of generators consisting of compact objects (\cite[Definition 1.7]{Neeman}). A subcategory $\mathrm T'$ of $\mathrm T$ is called {\bf localizing} if it is closed under coproducts in $\mathrm T.$ A triangulated category $\mathrm T$ with arbitrary coproducts is compactly generated with a set of compact generators $S$ if and only if $\mathrm T$ itself is the only localizing subcategory that contains $S$ (\cite[Lemma 2.21]{Schwede-Shipley}).

The category $\DM_{Nis}(\kk,R)$ is a compactly generated triangulated category with a set of compact generators $\mathcal G=\{M(X)(n)| X\in\Sm/k, n\in\Z \}$ (\cite[Th\'eor\`eme 4.5.67]{Ayoub}). For $n\in\Z,$ let $\mathcal G^\eff(n):=\{M(X)(p)| X\in\Sm/k, p\geq n\}.$ We define $\DM_{Nis}^\eff(\kk,R)(n)$ to be the smallest localizing subcategory that contains $\mathcal G^\eff(n)$ (\cite[3.1.5]{Pelaez}). Note that $\DM_{Nis}^\eff(\kk,R)(0)$ is equivalent to the category of effective motives $\DM_{Nis}^\eff(\kk,R)$ under the infinite suspension functor $\Sigma^\infty\colon \DM_{Nis}^\eff(\kk,R)\longrightarrow \DM_{Nis}(\kk,R).$ Here, $\DM_{Nis}^\eff(\kk,R)$ is the homotopy category of $\mathrm{Ch}(\mathrm{Sh}_{Nis}^\tr(\kk, R))$ with respect to the $\mathbb A^1$-local model structure, or equivalently, the Verdier localization of the derived category of the abelian category $\mathrm{Sh}_{Nis}^\tr(\kk, R)$ with respect to the class of morphisms $C_{\mathbb A^1}.$ From now on, we identify $\DM_{Nis}^\eff(\kk,R)(0)$ and $\DM_{Nis}^\eff(\kk,R).$ The canonical functor $M\colon \Sm/\kk\longrightarrow \DM_{Nis}(\kk,R)$ factors through $\Sigma^\infty$ by the constructions. We write the image of $X$ in $\DM_{Nis}^\eff(\kk,R)$ by the same symbol $M(X).$ This is by definition the image of $R_\tr(X)$ in $\DM_{Nis}^\eff(\kk,R).$

Consider the inclusion $i_n\colon \DM_{Nis}^\eff(\kk, R)(n)\longrightarrow \DM_{Nis}^\eff(\kk,R)$ ($n\geq0$). Since both target and source are compactly generated, by Neeman's Brown representability theorem (applied in the form of \cite[Theorem 2.1.3]{Pelaez}), the functor $i_n$ has a right adjoint $r_n\colon \DM_{Nis}^\eff(\kk,R)\longrightarrow \DM_{Nis}^\eff(\kk,R)(n),$ and $r_n$ is a triangulated functor. We write $f_n:=i_n\circ r_n$ and the counit of the adjunction is denoted by $\epsilon_n^M=\epsilon_n\colon f_nM\longrightarrow M.$ The functor $f_n$ is called the $(n-1)$-th effective cover and discussed further in \cite[Subsection 3.3]{Pelaez}.

The above constructions work for the \'etale topology as long as $\DM_{\et}^\eff(\kk,R)$ is compactly generated (for example, if $\kk$ has finite cohomological dimension or $R$ is a $\Q$-algebra). However, we do not use this fact because the following lemma does not hold for the \'etale topology. From now on, we write $M\langle n\rangle:=M(n)[2n]$ for short.

\begin{lem}\label{lem: Poincare}
Let $X$ be a smooth $\kk$-scheme of dimension $d$ and let $M^c(X)\in \DM_{Nis}^\eff(\kk,R)$ be the motive of $X$ with compact supports (\cite[Definition 16.13]{MVW}). Then, there is an isomorphism $M(X)\cong \inthom^\eff(M^c(X), R\langle d\rangle)$ in $\DM_{Nis}^\eff(\kk,R).$
\end{lem}

\begin{proof}
By \cite[Theorem 4.3.7 (3)]{VSF5}, there is an isomorphism $M^c(X)^\vee\cong M(X)\langle-d\rangle$ in $\DM_{Nis, gm}(\kk,R),$ where $M^c(X)^\vee$ is the dual of $M^c(X).$ Since $\inthom^\eff(M^c(X),R(d))$ is geometric by \cite[Corollary 20.4]{MVW} and a localization triangle (\cite[Theorem 16.5]{MVW} under resolution of singularities; \cite[Proposition 5.3.5]{Kelly} unconditionally for $\Z[1/p]$-coefficients), we have $M^c(X)^\vee=\inthom^\eff(M^c(X),R(d))(-d)$ in $\DM_{Nis,gm}(\kk,R)$ by the definition of dual objects. Therefore, there is an isomorphism $M(X)\cong\inthom^\eff(M^c(X),R\langle d\rangle)$ in $\DM_{Nis}^\eff(\kk,R).$
\end{proof}

The following definition is directly motivated by \cite[Corollary 5.3.3]{Pelaez}.

\begin{defn}\label{defn: filtration}
For a scheme $X\in\Sm/k$ of dimension $d,$ set 
\[F^nh_0^{Nis}(X)_R:=\ker\{h_0^{Nis}(X)_R\cong H^0(M(X))\cong H^0(\inthom^\eff(M^c(X), R\langle d\rangle))\buildrel \epsilon_{d+1-n}^*\over\longrightarrow H^0(\inthom^\eff(f_{d+1-n}M^c(X), R\langle d\rangle))\},\]
for $n\in\{0,1,\cdots d+1\},$ where $H^0$ is the cohomology with respect to the homotopy $t$-structure on $\DM_{Nis}^\eff(\kk,R).$ The second isomorphism is induced by the one in Lemma~\ref{lem: Poincare}.

A filtration $F^nh_0^\et(X)_R$ on $h_0^\et(X)_R$ is defined by \'etale sheafification $a_\et:$
\[F^nh_0^\et(X):=\mathrm{im}\{a_{\et}F^nh_0^{Nis}(X)_R\longrightarrow a_\et h_0^{Nis}(X)_R= h_0^\et(X)_R\}.\]
\end{defn}

If $R$ is a $\Q$-algebra, the sheafification $a_\et$ induces an equivalence of categories $\alpha\colon \DM_{Nis}^\eff(\kk,R)\longrightarrow \DM_\et^\eff(\kk,R)$ (\cite[Theorem 14.30]{MVW}). Hence, in this situation, the filtration on $h_0^\et(X)_R$ can be described as
\[F^nh_0^{\et}(X)_R:=\ker\{h_0^{\et}(X)_R\cong H^0(M(X))\cong H^0(\inthom^\eff(M^c(X), R\langle d\rangle))\buildrel\epsilon_{d+1-n}^*\over\longrightarrow H^0(\inthom^\eff(f_{d+1-n}M^c(X), R\langle d\rangle))\},\]
for $n\in\{0,1,\cdots d+1\},$ where $H^0$ is the cohomology with respect to the homotopy $t$-structure on $\DM_{\et}^\eff(\kk,R)$ and $\inthom^\eff$ is the internal hom in $\DM_{\et}^\eff(\kk,R).$

Let us show that our filtration satisfies (A) and (B) of Conjecture~\ref{conj: ABV} and their analogues in the Nisnevich topology.

\begin{prop}[cf. Conjecture~\ref{conj: ABV}(A)]\label{prop: A}
Let $X\in\Sm/\kk$ be a smooth scheme of dimension $d.$ Then, $F^nh_0^\tau(X)_R$ is a decreasing filtration on $h_0^\tau(X)_R$ such that $F^{0}h_0^\tau(X)_R=h_0^\tau(X)_R$ and $F^{d+1}h_0^\tau(X)_R=0.$
\end{prop}

\begin{proof}
It is enough to prove the claim for $\tau=Nis.$ The case $\tau=\et$ is immediate from this. Let $M\in \DM_{Nis}^\eff(\kk,R).$ By definition, $r_{n+1}M$ belongs to $\DM_{Nis}^\eff(\kk,R)(n+1),$ so \emph{a fortiori,} to $\DM_{Nis}^\eff(\kk,R)(n).$ Therefore, by the universal property of $\epsilon_n^M\colon f_nM\longrightarrow M,$ there is a unique morphism $h\colon f_{n+1}M\longrightarrow f_{n}M$ that satisfies $\epsilon_{n+1}^M=\epsilon_n^M\circ h.$ Apply this to $M=M^c(X).$ It follows that $F^{n+1}h_0^{Nis}(X)\subset F^{n}h_0^{Nis}(X).$ Hence, $F^n$ is a decreasing filtration.

For the triviality of the $(d+1)$-th filter, simply observe that $r_0$ is the identity as it is right adjoint to the identity $i_0$ on $\DM_{Nis}^\eff(\kk,R)(0)$($=\DM_{Nis}^\eff(\kk,R).$ 

To show that $F^{0}h_0^\tau(X)_R=h_0^\tau(X)_R,$ we need to show that $\epsilon_{d+1}\colon f_{d+1}M^c(X)\longrightarrow M^c(X)$ is trivial. We claim more generally that $M^c(X)\in \DM_{Nis}^\eff(\kk, R)(d+1)^\perp,$ where $\DM_{Nis}^\eff(\kk, R)(d+1)^\perp$ is the full subcategory of $\DM_{Nis}^\eff(\kk, R)$ consisting of objects $M$ such that $\Hom_{\DM_{Nis}^\eff(\kk, R)}(N,M)=0$ for every $N\in \DM_{Nis}^\eff(\kk, R)(d+1).$ For this, it suffices to show that $Hom_{\DM_{Nis}^\eff(\kk, R)}(M(Y)(d+r)[s], M^c(X))=0$ for each $Y\in\Sm/\kk, r\geq1$ and $s\in\Z$ (see \cite[Remark 2.1.2]{Pelaez}). By the Suslin-Friedlander duality (\cite[Theorem 8.2]{VSF5}), there is an isomorphism 
\[Hom_{\DM_{Nis}^\eff(\kk, R)}(M(Y)(d+r)[s], M^c(X))\cong Hom_{\DM_{Nis}^\eff(\kk, R)}(M(X\times Y)(r)[s-2d], R),\]
but the right hand side vanishes as shown in the proof of \cite[Lemma 5.1.1]{Pelaez}.
\end{proof}

\begin{prop}[cf. Conjecture~\ref{conj: ABV}(B)]\label{prop: B}
The filtration $F^nh_0^{\tau}(X)_R$ is compatible with the action of correspondences, i.e. for any $\gamma\in Cor(X,Y),$ the induced morphism of sheaves $h_0^{\tau}(X)_R\longrightarrow h_0^{\tau}(Y)_R$ is compatible with the filtration. 
\end{prop}

\begin{proof}
We only need to deal with the case $\tau=Nis.$ Let $d_X=\dim X$ and $d_Y=\dim Y.$ For any smooth scheme $X$ and $r\geq 0,$ there is a commutative diagram in $\DM_{Nis}^\eff(\kk,R)$
\begin{displaymath}
\xymatrix{\inthom^\eff(M^c(X), R\langle d_X\rangle)\ar[r] \ar[d]_\cong & \inthom^\eff(f_{d_X+1-n}M^c(X), R\langle d_X\rangle) \ar[d]_\cong \\
\inthom^\eff(M^c(X)\langle r\rangle, R\langle d_X+r\rangle) \ar[r]  \ar[dr] & \inthom^\eff((f_{d_X+1-n}M^c(X))\otimes R\langle r\rangle, R\langle d_X+r\rangle)\\
& \inthom^\eff(f_{d_X+1-n+r}(M^c(X)\langle r\rangle), R\langle d_X+r\rangle) \ar[u]^\cong_f,}
\end{displaymath}
where the vertical arrows in the square are isomorphisms by Voevodsky's cancellation theorem (\cite{Voevodsky cancellation}) and the isomorphism $f$ is induced by the isomorphism 
\[t_r^\eff(M^c(X))\colon(f_{d_X+1-n}M^c(X))\otimes R(r) \longrightarrow f_{d_X+1-n+r}(M^c(X)(r))\]
in \cite[Proposition 3.3.3 (2)]{Pelaez} (and shifting it by $2r$).

In particular, setting $r=d_Y,$ we see that 
\[F^nh_0^{Nis}(X)_R=\ker\{h_0^{Nis}(X)_R\cong H^0(M(X))\cong H^0(\inthom^\eff(M^c(X)\langle d_Y\rangle, R\langle d_X+d_Y\rangle)\]
\[\longrightarrow H^0(\inthom^\eff(f_{d_X+d_Y+1-n}(M^c(X)\langle d_Y\rangle), R\langle d_X+d_Y\rangle)\}.\]
Similarly, we have 
\[F^nh_0^{Nis}(Y)_R=\ker\{h_0^{Nis}(Y)_R\cong H^0(M(Y))\cong H^0(\inthom^\eff(M^c(Y)\langle d_X\rangle, R\langle d_X+d_Y\rangle)\]
\[\longrightarrow H^0(\inthom^\eff(f_{d_X+d_Y+1-n}(M^c(Y)\langle d_X\rangle), R\langle d_X+d_Y\rangle)\}.\]

Now, a finite correspondence $\gamma\in \Cor(X,Y)$ induces a morphism $M(X)\longrightarrow M(Y)$ in $\DM_{Nis}^\eff(\kk,R),$ and this induces the morphism in question $h_0^{Nis}(X)_R\longrightarrow h_0^{Nis}(Y)_R.$ Thus, the commutativity of the following diagram implies the  proposition:
\begin{displaymath}
\xymatrix{M(X) \ar[d]_\gamma \ar[r] & \inthom^\eff(M^c(X)\langle d_Y\rangle, R\langle d_X+d_Y\rangle) \ar[r] \ar[d]_{g^*} & \inthom^\eff(f_{d_X+d_Y+1-n}(M^c(X)\langle d_Y\rangle), R\langle d_X+d_Y\rangle) \ar[d]_{f_{d_X+d_Y+1-n}(g)^*}\\
M(Y) \ar[r] & \inthom^\eff(M^c(Y)\langle d_X\rangle, R\langle d_X+d_Y\rangle) \ar[r]  & \inthom^\eff(f_{d_X+d_Y+1-n}(M^c(Y)\langle d_X\rangle), R\langle d_X+d_Y\rangle),}
\end{displaymath}
where $g$ is the composition
\[M^c(Y)\langle d_X\rangle \cong \inthom^\eff(M(Y),R\langle d_X+d_Y\rangle)\buildrel\gamma^*\over\longrightarrow \inthom^\eff(M(X),R\langle d_X+d_Y\rangle)\cong M^c(X)\langle d_Y\rangle\]
of the canonical isomorphisms and the composition with $\gamma.$ (In other words, $g$ is the $(d_X+d_Y)$-twist and $2(d_X+d_Y)$-shift of the morphism in (non-effective) $\DM_{Nis, gm}(\kk,R)$
\[M^c(Y)\langle-d_Y\rangle\cong M(Y)^\vee\buildrel\gamma^\vee\over \longrightarrow M(X)^\vee\cong M^c(X)\langle -d_X\rangle,\]
where $\vee$ stands for the dual in $\DM_{Nis,gm}(\kk,R).$ As the left square is in the category of geometric motives $\DM_{Nis, gm}(\kk,R)$ (\cite[Theorem 16.15, Corollary 20.4]{MVW}; \cite[Proposition 5.3.5]{Kelly} to remove the hypothesis of resolution of singularities), the commutativity of the square is immediate from this description of $g.$)\end{proof}


\subsection{Under Ayoub's conjectures on $n$-motives}

We derive properties of the filtration $F^nh_0^\et(X)$ from Ayoub's conjectures on $n$-motives. In the end, under these conjectures, we conclude that if the filtration in Conjecture~\ref{conj: ABV} exists, it agrees with the one in Definition~\ref{defn: filtration}. In this subsection, $\tau$ is either $Nis$ or $\et.$

We write $\DM_{\tau,\leq n}^\eff(\kk,R)$ for the smallest localizing subcategory of $\DM_{\tau}^\eff(\kk,R)$ that contains the set of objects $\{R_\tr(X)| X\in(\Sm/\kk)_{\leq n}\}.$ The category $\DM_{\tau,\leq n}^\eff(\kk,R)$ is called the triangulated category of $n$-motives. The conjectures of Ayoub are concerned with motives with coefficients in a $\Q$-algebra $R.$

\begin{conj}[{\cite[Conjecture 4.22]{AyoubOpen}}]\label{conj: contraction functor}
Let $R$ be a $\Q$-algebra. Then, the functor 
\[\inthom^\eff(R(1),-)\colon \DM_\et^\eff(\kk,R)\longrightarrow \DM_\et^\eff(\kk,R)\]
takes $\DM_{\et,\leq n}^\eff(\kk,R)$ to $\DM_{\et,\leq n-1}^\eff(\kk,R),$ where $\inthom^\eff$ stands for the internal hom in $\DM_\et^\eff(\kk,R)$ and $R(1)$ is the Tate motive. 
\end{conj}

\begin{conj}[{\cite[Conjecture 4.27]{AyoubOpen}, \cite[Conjecture 2.5.3]{ABV}}]\label{conj: heart}
Let $R$ be a $\Q$-algebra. Then, the homotopy $t$-structure on $\DM_\et^\eff(\kk,R)$ restricts to a $t$-structure on $\DM_{\et,\leq n}^\eff(\kk,R),$ and the heart of this is the category $\HI_{\et, \leq n}^\tr(\kk,R)$ of $n$-motivic sheaves.
\end{conj}

\begin{rem}\label{rem: sort out}
In \cite[Definition 4.19]{AyoubOpen}, an $n$-motivic sheaf of $R$-modules is called an $n$-presented homotopy invariant sheaf with transfers. The latter is defined as a sheaf $F\in \HI_\et^\tr(\kk,R)$ such that there is an exact sequence of sheaves in $\HI_{\et}^\tr(\kk,R)$
\[\bigoplus_{j\in J}h_0^\et(Y_j)_R\longrightarrow \bigoplus_{i\in I}h_0^\et(Y_i)_R\longrightarrow F\longrightarrow 0,\]
where $X_i$ and $Y_j$ are objects in $(\Sm/\kk)_{\leq n}.$ As remarked in [\emph{ibid., Remark 4.20}], the two notions agree, so Conjecture~\ref{conj: heart} stands in the present context. Indeed, by Proposition~\ref{prop: proved facts on n-motivic sheaves}, any $n$-presented homotopy invariant sheaf with transfers is an $n$-motivic sheaf. Conversely, by Remark~\ref{rem: n-generated sheaves}, any $n$-motivic \'etale sheaf $F$ is the $h_0^\et$ of some strongly $n$-generated sheaf $G.$ Now, by the description of $\sigma_n^*$ recalled in Subsection~\ref{subsection: $n$-motivic sheaves}, we have $G\cong \colim_{X\to \sigma_{n*}G}R_\tr(X),$ where the colimit is taken over the category $Cor(\kk_{\leq n})/\sigma_{n*}G.$ Since $h_0^\et$ commutes with colimits, we obtain
\[F\cong h_0^\et(G)\cong h_0^\et(\colim_{X\to \sigma_{n*}G}R_\tr(X))\cong \colim_{X\to \sigma_{n*}G}h_0^\et(R_\tr(X))= \colim_{X\to \sigma_{n*}G}h_0^\et(X)_R.\]
\end{rem}

The functor $r_n$ in Subsection~\ref{subsection: Construction of the filtration} can be expressed in terms of internal hom in $\DM_\et^\eff(\kk,R).$ This description of $r_n$ enables us to use the power of Conjectures~\ref{conj: contraction functor} and \ref{conj: heart}. The author learned this method from \cite[6.1.9]{Pelaez}.

\begin{prop}[{\cite[Proposition 1.1]{Huber-Kahn}}]\label{prop: link}
Suppose $n\geq0.$ The functor $\nu_n\colon \DM_\tau^\eff(\kk,R)\longrightarrow \DM_\tau^\eff(\kk,R)(n)$ defined by $\nu_n(M)=\inthom^\eff(R(n),M)(n)$ is right adjoint to the inclusion $i_n,$ i.e. $\nu_n\cong r_n.$
\end{prop}

\begin{proof}
The proof in \cite[Proposition 1.1]{Huber-Kahn} works verbatim. (Let us remark that the proof uses Voevodsky's cancellation theorem \cite{Voevodsky cancellation}.)
\end{proof}

Here is a (very) weak version of (D).

\begin{prop}\label{prop: embeddable}
Let $R$ be a $\Q$-algebra. Assume Conjectures~\ref{conj: contraction functor} and \ref{conj: heart}. Then, for any $X\in \Sm/\kk,$ the sheaf $h_0^\et(X)_R/F^nh_0^\et(X)_R$ can be embedded into an $(n-1)$-motivic sheaf.
\end{prop}

\begin{proof}
Let $d:=\dim X.$ By Definition~\ref{defn: filtration} and Conjecture~\ref{conj: heart}, it is enough to prove that $\inthom^\eff(f_{d+1-n}M^c(X), R\langle d\rangle)$ is $(n-1)$-motivic. By Proposition~\ref{prop: link}, we have 
\[\inthom^\eff(f_{d+1-n}M^c(X), R\langle d\rangle)\cong\inthom^\eff(\inthom^\eff(R(d+1-n),M^c(X))(d+1-n), R\langle d\rangle).\]
Since the right hand side is isomorphic to $\inthom^\eff(R(d+1-n), \inthom^\eff(\inthom^\eff(R(d+1-n),M^c(X)),R\langle d\rangle))$ and since we are under Conjecture~\ref{conj: contraction functor}, it remains to show that $\inthom^\eff(\inthom^\eff(R(d+1-n),M^c(X)),R\langle d\rangle)$ is $d$-motivic. Now, $R\langle d\rangle$ is $d$-motivic because there is a decomposition $M(\mathbb P_\kk^d)\cong\bigoplus_{i=0}^d R\langle i\rangle$ (\cite[Exercise 15.11]{MVW}). Therefore, it suffices to show the following lemma.
\end{proof}

\begin{lem}
Let $R$ be a $\Q$-algebra. Assume Conjecture~\ref{conj: contraction functor}. For any motive $M\in\DM_\et^\eff(\kk,R)$ and any $N\in \DM_{\et,\leq n}^\eff(\kk,R),$ the internal hom $\inthom^\eff(M,N)$ belongs to $\DM_{\et,\leq n}^\eff(\kk,R).$
\end{lem}

\begin{proof}
This is \cite[Proposition 4.26]{AyoubOpen} if $M$ belongs to $\DM_{\et,gm}^\eff(\kk,R).$ By \cite[Theorem 2.1.3]{Pelaez}, the inclusion $j_n\colon\DM_{\et,\leq n}^\eff(\kk,R)\longrightarrow  \DM_\et^\eff(\kk,R)$ has a right adjoint $d_n\colon \DM_\et^\eff(\kk,R) \longrightarrow \DM_{\et,\leq n}^\eff(\kk,R),$ which is a triangulated functor, and there is another triangulated functor $l_n\colon \DM_{\et}^\eff(\kk,R) \longrightarrow \DM_{\et}^\eff(\kk,R)$ together with a natural triangle in $\DM_{\et}^\eff(\kk,R)$ for any $E\in\DM_{\et}^\eff(\kk,R):$
\[j_nd_n E\longrightarrow E\longrightarrow l_nE\buildrel[1]\over\longrightarrow \]
with $l_nE\in \DM_{\et,\leq n}^\eff(\kk,R)^\perp.$

Applying this to $E=\inthom^\eff(M,N),$ we see that it suffices to show that $l_n\inthom^\eff(M,N)=0$ for all effective motive $M.$ This is known under Conjecture~\ref{conj: contraction functor} for effective geometric motives by \cite[Proposition 4.26]{AyoubOpen}. Therefore, it suffices to show that the full triangulated subcategory $\mathrm{T}$ of $\DM_\et^\eff(\kk,R)$ consisting of objects $X$ such that $l_n\inthom^\eff(X,N)=0$ is localizing. 

Let $\mathrm{S}$ be the full triangulated subcategory of $DM_{\et,\leq n}^\eff(\kk,R)^\perp$ consisting of objects $E$ such that $d_nE=0.$ It follows from the above distinguished triangle that $l_n$ has a image in $\mathrm{S}.$ Let us write $l_n'\colon \DM_\et^\eff(\kk,R)\longrightarrow \mathrm{S}$ for the functor induced by $l_n.$ We claim that $l_n'$ is left adjoint to the inclusion $\mathrm{S}\longrightarrow \DM_\et^\eff(\kk,R).$ (This claim is a variant of \cite[Corollary 1.4 (ii)]{Huber-Kahn}.) Indeed, for any $E\in \DM_\et^\eff(\kk,R)$ and any $F\in \mathrm{S},$ the distinguished triangle for $E$ gives rise to an exact sequence of hom groups in $\DM_\et^\eff(\kk,R)$
\[\cdots\longrightarrow \Hom(j_nd_nE[1],F)\longrightarrow \Hom(l_n'E,F)\longrightarrow \Hom(E,F)\longrightarrow \Hom(j_nd_nE, F)\longrightarrow\cdots.   \]
Since $F$ belongs to $ \DM_{\et,\leq n}^\eff(\kk,R)^\perp,$ we have $\Hom(j_nd_nE[1],F)= \Hom(j_nd_nE, F)=0.$ Thus, the middle map $\Hom_{\mathrm{S}}(l_n'E,F)=\Hom_{\DM_\et^\eff(\kk,R)}(l_n'E,F)\longrightarrow \Hom_{\DM_\et^\eff(\kk,R)}(E,F)$ is an isomorphism. Therefore, $l_n'\colon \DM_\et^\eff(\kk, R)\longrightarrow \mathrm{S}$ is a left adjoint functor. In particular, it commutes with coproducts. This immediately implies that $\mathrm{T}$ is a localizing subcategory.
\end{proof}

It is shown that if there are filtrations $F^nh_0^\et(X)$ for all $X\in Sm/\kk$ that satisfies (A), (B) and the Weaker Version of (D) in Conjecture~\ref{conj: ABV}, then any homotopy invariant subsheaf of an $n$-motivic sheaf is again $n$-motivic (\cite[Corollary 1.4.5]{ABV}). Therefore, we have the following.

\begin{cor}\label{cor: D}
Suppose $R$ is a $\Q$-algebra and assume Conjectures~\ref{conj: contraction functor} and \ref{conj: heart}. If the filtration as in Conjecture~\ref{conj: ABV} exists, then the filtration in Definition~\ref{defn: filtration} satisfies Conjecture~\ref{conj: ABV}(D). 
\end{cor}

Let us define a filtration on an arbitrary homotopy invariant $\tau$-sheaf $S$ of $R$-modules with transfers. $S$ can be written as a colimit of sheaves $h_0^\tau(X)_R$ with $X\in\Sm/k$: $S\cong \colim_{\Cor(k)/S}h_0^\tau(X)_R$ (\cite[Corollary 1.1.8]{ABV}). We define a filtration on $S$ as follows. First, set 
\[S^{\leq i}:=\colim_{\Cor(k)/S}h_0^\tau(X)_R/F^{i+1}h_0^\tau(X)_R.\] 
Then, define 
\[F^iS:=\ker\{S\longrightarrow S^{\leq i-1}\}.\]   
This is clearly compatible with the filtration on $h_0^\tau(X)_R$ in Definition~\ref{defn: filtration}. Note also that $F^0S=S$ holds for any $S\in\HI_{\tau}^\tr(\kk, R).$

\begin{lem}\label{lem: filtration on HI sheaves}
Any morphism $S\longrightarrow S'$ in $\HI_\tau^\tr(\kk,R)$ respects the filtration. If $S$ is an $n$-motivic $\tau$-sheaf, then $F^iS=0$ for $i\geq n+1.$
\end{lem}

\begin{proof}
Any morphism respects the filtration because morphisms between sheaves of the form $h_0^\tau(X)_R$ ($X\in\Sm/k$) have this property by Proposition~\ref{prop: B}. For the triviality of the filtration on $n$-motivic sheaves for $i\geq n+1,$ note that we have $S\cong \colim_{\Cor(k_{\leq n})/S}h_0^\et(X)_R$ by Remark~\ref{rem: sort out}. With Propositions~\ref{prop: A} and \ref{prop: B}, this means that for $j\geq n$ we have
\[ S^{\leq j}=\colim_{\Cor(k)/S}h_0^\tau(X)_R/F^{j+1}h_0^\tau(X)_R\cong \colim_{\Cor(k_{\leq n})/S}h_0^\tau(X)_R/F^{j+1}h_0^\tau(X)_R =\colim_{\Cor(k_{\leq n})/S}h_0^\tau(X)_R=S.\]
Hence, $F^iS=\ker\{S\buildrel id\over\longrightarrow S\}=0$ for $i\geq n+1.$
\end{proof}

\begin{prop}\label{prop: universality in quotation}
Let $X$ be an arbitrary smooth $\kk$-scheme. Then, any morphism $f\colon h_0^\tau(X)_R\longrightarrow F$ of homotopy invariant sheaves of $R$-modules with transfers with $F\in\HI_{\tau,\leq n-1}^\tr(\kk,R)$ factors through $h_0^\tau(X)_R\longrightarrow h_0^\tau(X)_R/F^nh_0^\tau(X)_R.$  
\end{prop}

\begin{proof}
By Lemma~\ref{lem: filtration on HI sheaves}, $f$ respects the filtration $F^i,$ so the image of $F^nh_0^\tau(X)_R$ under $f$ belongs to $F^nF,$ but the latter is trivial again by the lemma.
\end{proof}

\begin{rem}
Proposition~\ref{prop: universality in quotation} is unconditional. When $\tau=\et$ and $R=\Q,$ it may be regarded as a very weak form of the conjectured existence of the left adjoint to the inclusion $\HI_{\et,\leq n-1}^\tr(\kk,\Q)\longrightarrow \HI_\et^\tr(\kk,\Q).$
\end{rem}

\begin{thm}\label{thm}
Suppose $R$ is a $\Q$-algebra and assume Conjectures~\ref{conj: contraction functor} and \ref{conj: heart}. If the filtration in Conjecture~\ref{conj: ABV} exists, then it agrees with the one in Definition~\ref{defn: filtration}.
\end{thm}

\begin{proof}
With Proposition~\ref{prop: A}, Proposition~\ref{prop: B} and Corollary~\ref{cor: D}, it remains to show the property (C) in Conjecture~\ref{conj: ABV}. Slightly more strongly, we show that if $U$ is a dense open subscheme of $X\in \Sm/\kk,$ then the morphism $F^nh_0^\et(U)_R\longrightarrow F^nh_0^\et(X)_R$ is surjective. We follow the proof of \cite[Proposition 1.4.6]{ABV}, especially the last part.

Consider the commutative diagram 
\begin{displaymath}
\xymatrix{ 0\ar[r] & F^nh_0^\et(U)_R \ar[r] \ar[d]_a & h_0^\et(U)_R  \ar[r] \ar[d]_b & h_0^\et(U)_R/ F^nh_0^\et(U)_R \ar[r] \ar[d]_c & 0 \\
0\ar[r] & F^nh_0^\et(X)_R \ar[r] & h_0^\et(X)_R  \ar[r] & h_0^\et(X)_R/ F^nh_0^\et(X)_R \ar[r] & 0. }
\end{displaymath}
We shall show that $\mathrm{coker}~a$ is trivial. But for this, it is enough to prove that $\mathrm{coker}~a$ is $(n-1)$-motivic. Indeed, the kernel-cokernel sequence associated with the composition $F^nh_0^\et(U)_R\longrightarrow F^nh_0^\et(X)_R\longrightarrow h_0^\et(X)_R$
gives rise to an exact sequence
\[0\longrightarrow \mathrm{coker}~a\longrightarrow h_0^\et(X)_R/F^nh_0^\et(U)_R\buildrel f\over\longrightarrow h_0^\et(X)_R/F^nh_0^\et(X)_R\longrightarrow 0.\]
By Corollary~\ref{cor: D}, $h_0^\et(X)_R/F^nh_0^\et(X)_R$ is $(n-1)$-motivic. If $\mathrm{coker}~a$ is $(n-1)$-motivic, $h_0^\et(X)_R/F^nh_0^\et(U)_R$ is also $(n-1)$-motivic by Proposition~\ref{prop: proved facts on n-motivic sheaves}(i). Thus, Proposition~\ref{prop: universality in quotation} implies that $f$ is an isomorphism; hence $\mathrm{coker}~a$ is trivial.

Let us show that $\mathrm{coker}~a$ is $(n-1)$-motivic. The map $b$ in the diagram is an epimorphism by \cite[Corollary 22.8]{MVW} and Yoneda's lemma. Therefore, $\mathrm{coker}~a$ is a subquotient of $h_0^\et(U)_R/ F^nh_0^\et(U)_R.$ Because, if the filtration in Conjecture~\ref{conj: ABV} exists, $\HI_{\et,\leq n-1}^\tr(\kk, R)$ is a Serre subcategory of $\HI_{\et}^\tr(\kk, R)$ (\cite[Corollary 1.4.5]{ABV}), it is enough to show that $h_0^\et(U)_R/ F^nh_0^\et(U)_R$ is $(n-1)$-motivic. But this is Corollary~\ref{cor: D}.
\end{proof}


\end{document}